\documentclass[a4paper]{amsart}
\usepackage{amsthm,amsfonts,amsmath,amssymb}
\usepackage[abs]{overpic}

\newtheorem{theorem}{Theorem}[section]
\newtheorem{lemma}[theorem]{Lemma}

\newtheorem{proposition}[theorem]{Proposition}
\newtheorem{corollary}[theorem]{Corollary}

\theoremstyle{definition}
\newtheorem{definition}[theorem]{Definition}

\theoremstyle{remark}
\newtheorem{remark}[theorem]{Remark}

\newtheorem{claim}[theorem]{Claim}

\newcommand{\omu}{\overline{\mu}}
\newcommand{\e}{\varepsilon}
\newcommand{\de}{\delta}

\makeatletter

\@addtoreset{figure}{section}
\makeatother

\makeatletter
  
  \@addtoreset{equation}{section}
\makeatother

\begin{document}
\title[Classification of string links up to $2n$-moves and link-homotopy]{Classification of string links\\ up to $2n$-moves and link-homotopy}

\author[Haruko A. Miyazawa]{Haruko A. Miyazawa}
\address{Institute for Mathematics and Computer Science, Tsuda University,
2-1-1 Tsuda-Machi, Kodaira, Tokyo, 187-8577, Japan}
\curraddr{}
\email{aida@tsuda.ac.jp}
\thanks{}

\author[Kodai Wada]{Kodai Wada}
\address{Faculty of Education and Integrated Arts and Sciences, Waseda University, 1-6-1 Nishi-Waseda, Shinjuku-ku, Tokyo, 169-8050, Japan}
\curraddr{}
\email{k.wada8@kurenai.waseda.jp}
\thanks{The second author was supported by a Grant-in-Aid for JSPS Research Fellow (\#17J08186) of the Japan Society for the Promotion of Science.}

\author[Akira Yasuhara]{Akira Yasuhara}
\address{Faculty of Commerce, Waseda University, 1-6-1 Nishi-Waseda, Shinjuku-ku, Tokyo, 169-8050, Japan}
\curraddr{}
\email{yasuhara@waseda.jp}
\thanks{The third author was partially supported by a Grant-in-Aid for Scientific Research (C) (\#17K05264) of the Japan Society for the Promotion of Science and a Waseda University Grant for Special Research Projects (\#2018S-077).}






\begin{abstract}
Two string links are equivalent up to $2n$-moves and link-homotopy if and only if their all Milnor link-homotopy invariants are congruent modulo~$n$. 
Moreover, the set of the equivalence classes forms a finite group generated by elements of order~$n$. 
The classification induces that if two string links are equivalent up to $2n$-moves for every $n>0$, 
then they are link-homotopic. 
\end{abstract}

\maketitle

\section{Introduction} 
In 1950s, J.~Milnor~\cite{M54,M57} defined a family of link invariants, known as {\em Milnor $\omu$-invariants}. 
For an ordered oriented $m$-component link $L$ in the $3$-sphere $S^3$, the {\em Milnor number $\mu_{L}(I)$} $(\in\mathbb{Z})$ of $L$ is specified by a finite sequence $I$ of elements in $\{1,\ldots,m\}$. 
This number is only well-defined up to a certain indeterminacy $\Delta_L(I)$, 
i.e. the residue class $\omu_{L}(I)$ of $\mu_{L}(I)$ modulo $\Delta_L(I)$ is a link invariant. 
The invariant $\omu_{L}(ij)$ for a sequence $ij$ is just the linking number between the $i$th and $j$th component of $L$. 
This justifies regarding $\omu$-invariants  as ``generalized linking numbers''. 

In~\cite{HL} N.~Habegger and X.-S.~Lin defined Milnor numbers for {\em string links} and  
proved that Milnor numbers are well-defined invariants without taking modulo. 
These numbers are called {\em Milnor $\mu$-invariants}. 
It is remarkable that $\mu$-invariants for non-repeated sequences classify string links up to link-homotopy~\cite{HL} (whereas $\omu$-invariants are not enough strong to classify links with four or more components up to link-homotopy~\cite{L}). 
Here the {\em link-homotopy}, introduced by Milnor in~\cite{M54}, is the equivalence relation on (string) links 
generated by self-crossing changes and ambient isotopies.
In addition to link-homotopy, there are various ``geometric'' ~equivalence relations on (string) links that
are related to Milnor invariants, 
e.g. concordance~\cite{S,Casson}, (self) $C_{k}$-equivalence~\cite{H,FY,Yagt,Ytrans,MY10} and Whitney tower concordance~\cite{CST1,CST2,CST3},~etc.

A {\em $2n$-move} is a local move illustrated in Figure~\ref{n-move}, and the {\em $2n$-move equivalence} is the equivalence relation generated by $2n$-moves and ambient isotopies. 
The $2n$-moves were probably first studied by S.~Kinoshita in 1957~\cite{K57}. 
It is known that several $2n$-move equivalence invariants are derived from polynomial invariants, 
the Alexander~\cite{K80}, Jones, Kauffman and HOMFLYPT polynomials~\cite{P}. 
Besides polynomial invariants, Fox colorings and Burnside groups give $2n$-move equivalence invariants~\cite{DP02,DP04}.

\begin{figure}[htbp]
  \begin{center}
    \begin{overpic}[width=11cm]{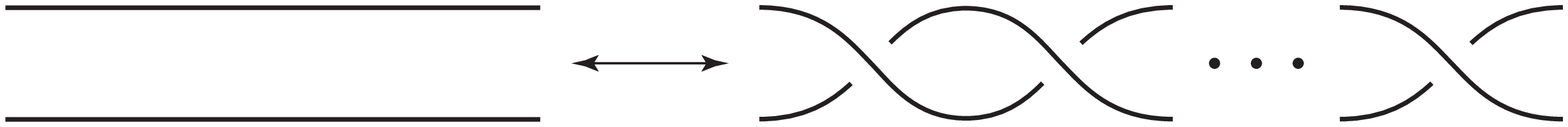}
      \put(172,-2){{\footnotesize $1$}}
      \put(210.5,-2){{\footnotesize $2$}}
      \put(287,-2){{\footnotesize $2n$}}
    \end{overpic}
  \end{center}
  \caption{$2n$-move}
  \label{n-move}
\end{figure}

Both Milnor invariants and $2n$-moves are well-studied in Knot Theory. 
However, to the best of the authors' knowledge, there are no research articles relating Milnor invariants and $2n$-moves 
(except for the easily observed fact that the linking numbers modulo $n$ are $2n$-move equivalence invariants). 
In this paper, we show the following theorem that establishes an unexpected relationship between 
Milnor link-homotopy invariants and $2n$-moves.

\begin{theorem}\label{th-sl}
Let $n$ be a positive integer. 
Two string links $\sigma$ and $\sigma'$ are $(2n+{\rm lh})$-equivalent if and only if 
$\mu_{\sigma}(I)\equiv\mu_{\sigma'}(I)\pmod{n}$ for any non-repeated sequence~$I$. 
\end{theorem}

\noindent
Here, the {\em $(2n+{\rm lh})$-equivalence} is the equivalence relation generated by $2n$-moves, self-crossing changes and ambient isotopies. 
Note that ``$2n+{\rm lh}$'' stands for the combination of $2n$-move equivalence and link-homotopy.
In order to prove Theorem~\ref{th-sl}, we give  a complete list of representatives for string links up to $(2n+{\rm lh})$-equivalence (Proposition~\ref{prop-rep-2nlh}). 

Let $\mathcal{SL}(m)$ denote the set of $m$-component string links. 
Since the set of link-homotopy classes of $\mathcal{SL}(m)$ forms a group \cite{HL}, 
it is not hard to see that the set of $(2n+\rm{lh})$-equivalence classes  
is also a group. Moreover we have the following. 

\begin{corollary}\label{cor-group}
The set of $(2n+\rm{lh})$-equivalence classes of $\mathcal{SL}(m)$ forms a finite group generated by elements of order~$n$, 
and the order of the group is $n^{s_{m}}$, where $s_{m}=\sum_{r=2}^{m}(r-2)!\binom{m}{r}$. 
\end{corollary}

\noindent
The link-homotopy, concordance and $C_k$-equivalence give group structures on those equivalence classes of $\mathcal{SL}(m)$, 
respectively~\cite{HL, H}. 
The set of link-homotopy classes is a torsion free group of rank $s_m$ (see~\cite[Section 3]{HL}), and 
the concordance classes contain elements of order 2. 
It is still open if the concordance classes contain elements of order $\geq 3$ and 
if the $C_k$-equivalence classes have torsion elements. 
In contrast to these facts, Corollary~\ref{cor-group} implies that, for any integer $n\geq2$, 
the $(2n+\rm{lh})$-equivalence classes contain elements of order $n$. 
 
As a consequence of Theorem~\ref{th-sl}, we obtain a necessary and sufficient condition for which 
a link in $S^3$ is $(2n+\rm{lh})$-equivalent to the trivial link by means of Milnor {\em numbers}.

\begin{corollary}[Corollary~\ref{cor-trivial}]
Let $n$ be a positive integer.
An $m$-component link $L$ in $S^3$ is $(2n+{\rm lh})$-equivalent to the trivial link if and only if 
$\mu_L(I)\equiv0\pmod{n}$ for any non-repeated sequence~$I$. 
\end{corollary}

In~\cite{Fox}, R.~H.~Fox introduced the notion of {\em congruence classes modulo $(n,q)$} of knots in $S^{3}$ for integers $n>0$ and $q\geq 0$, 
and asked whether the set of congruence classes of a knot determines the knot type. 
More precisely, he asked the following question:  
If two knots are congruent modulo $(n,q)$ for every $n$ and $q$, then are they ambient isotopic? 
We note that the notion of congruences and the question can be extended to (string) links. 
It is known in~\cite{Fox,NS,N,La} that the Alexander and Jones polynomials restrict the possible congruence classes. 
In particular, M.~Lackenby proved that if two links are congruent modulo $(n,2)$ for every $n$, 
then they have the same Jones polynomial~\cite[Corollaly 2.4]{La}.

Since the $2n$-move equivalence implies the congruence modulo~$(n,2)$, 
it would be interesting to ask whether the set of $2n$-move equivalence classes of a (string) link determines the link type. 
Theorem~\ref{th-sl} implies that if two string links are $2n$-move equivalent for every $n$, 
then they share all Milnor invariants for non-repeated sequences. 
Combining this and 
the classification of string links up to link-homotopy~\cite{HL}, 
we have the following corollary.

\begin{corollary} 
If two string links are $2n$-move equivalent for every $n$, then they are link-homotopic. 
In particular, if a $($string$)$ link $L$ is $2n$-move equivalent to the trivial one for every $n$, then $L$ is link-homotopically trivial. 
\end{corollary}

\section{Preliminaries}\label{sec-sl}
In this section, we summarize the definitions of string links and their Milnor invariants from~\cite{M57,F,HL,Yagt}.

\subsection{String links and Milnor $\mu$-invariants}
Let $\mathbb{D}^{2}$ be the unit disk in the plane equipped with $m$ points $x_{1},\ldots,x_{m}$ in its interior, 
lying in order on the $x$-axis. 
Let $I_{1},\ldots,I_{m}$ be $m$ copies of $[0,1]$. 
An {\em $m$-component string link} is the image of a proper embedding 
\[
\bigsqcup_{i=1}^{m}I_{i}\longrightarrow \mathbb{D}^{2}\times [0,1]
\] 
such that the image of each $I_{i}$ runs from $(x_{i},0)$ to $(x_{i},1)$. 
Each strand of a string link is oriented upward. 
The $m$-component string link $\{x_{1},\ldots,x_{m}\}\times [0,1]$ in $\mathbb{D}^{2}\times [0,1]$ is called the {\em trivial $m$-component string link}, and is denoted by $\mathbf{1}_{m}$. 

Given an $m$-component string link $\sigma$, 
let $G(\sigma)$ denote the fundamental group of the complement $(\mathbb{D}^{2}\times [0,1])\setminus\sigma$ with a base point on the boundary of $\mathbb{D}^{2}\times\{0\}$, 
and let $G(\sigma)_{q}$ denote the $q$th term of the lower central series of $G(\sigma)$. 
Let $\alpha_{i}$ and $l_{i}$ be the $i$th meridian and the $i$th longitude of $\sigma$, respectively, illustrated in Figure~\ref{peripheral}. 
Abusing notation, we still denote by $\alpha_{i}$ the image of $\alpha_{i}$ in the $q$th nilpotent quotient $G(\sigma)/G(\sigma)_{q}$. 
We assume that each $l_{i}$ is the {\em preferred} longitude, i.e. the zero-framed parallel copy of the $i$th component of $\sigma$. 
Since $G(\sigma)/G(\sigma)_{q}$ is generated by $\alpha_{1},\ldots,\alpha_{m}$ (see~\cite{C,S}), 
the $i$th longitude $l_{i}$ is expressed as a word in $\alpha_{1},\ldots,\alpha_{m}$ for each $i\in\{1,\ldots,m\}$. 
We denote by $\lambda_{i}$ this word.

\begin{figure}[htbp]
  \begin{center}
    \begin{overpic}[width=10.5cm]{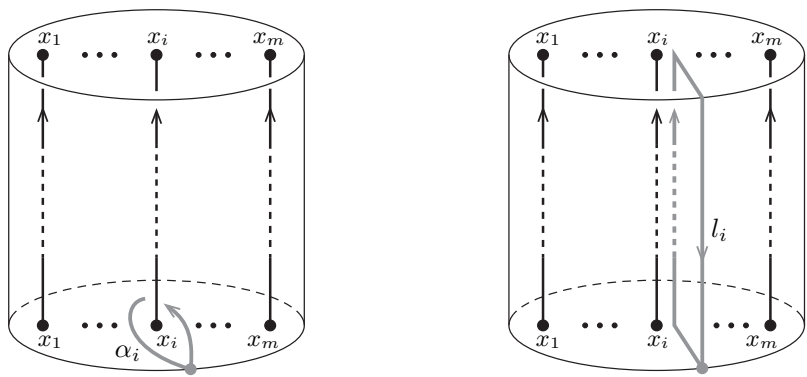}
      \put(11,125.5){{\small $x_{1}$}}
      \put(52,125.5){{\small $x_{i}$}}
      \put(91.5,125.5){{\small $x_{m}$}}
      \put(11,11){{\small $x_{1}$}}
      \put(55.5,11){{\small $x_{i}$}}
      \put(89.5,11){{\small $x_{m}$}}
      \put(40,7){$\alpha_{i}$}
      \put(197.5,125.5){{\small $x_{1}$}}
      \put(239,125.5){{\small $x_{i}$}}
      \put(278,125.5){{\small $x_{m}$}}
      \put(197.5,11){{\small $x_{1}$}}
      \put(239,11){{\small $x_{i}$}}
      \put(276,11){{\small $x_{m}$}}
      \put(263,52){$l_{i}$}
    \end{overpic}
  \end{center}
  \caption{The $i$th meridian $\alpha_{i}$ and the $i$th longitude $l_{i}$}
  \label{peripheral}
\end{figure}

Let $\langle\alpha_{1},\ldots,\alpha_{m}\rangle$ denote the free group on $\{\alpha_{1},\ldots,\alpha_{m}\}$, and let 
$\mathbb{Z}\langle\langle X_{1},\ldots,X_{m}\rangle\rangle$ denote the ring of formal power series in non-commutative variables $X_{1},\ldots,X_{m}$ with integer coefficients. 
The {\em Magnus expansion} is a homomorphism 
\[
E:\langle\alpha_{1},\ldots,\alpha_{m}\rangle
\longrightarrow \mathbb{Z}\langle\langle X_{1},\ldots,X_{m}\rangle\rangle
\]
defined by, for $1\leq i\leq m$: 
\[
E(\alpha_{i})=1+X_{i}, \ E(\alpha_{i}^{-1})=1-X_{i}+X_{i}^{2}-X_{i}^{3}+\cdots. 
\]
Let $I=j_{1}j_{2}\ldots j_{k}i$ $(k<q)$ be a sequence of elements in $\{1,\ldots,m\}$. 
The coefficient of $X_{j_{1}}\cdots X_{j_{k}}$ in the Magnus expansion $E(\lambda_{i})$ is called the {\em Milnor $\mu$-invariant} for the sequence $I$ and is denoted by $\mu_{\sigma}(I)$~\cite{HL}. 
(We define $\mu_{\sigma}(i)=0$.)
The length $|I|$ $(=k+1)$ of $I$ is called the {\em length} of $\mu_{\sigma}(I)$.

\subsection{Milnor's algorithm}\label{subsec-algorithm}
To compute $\mu_{\sigma}(I)$ we need to obtain the word $\lambda_{i}$ in $\alpha_{1},\ldots,\alpha_{m}$ concretely, which represents the $i$th longitude $l_{i}$. 
In~\cite{M57}, Milnor introduced an algorithm to give $\lambda_{i}$ by using the Wirtinger presentation of $G(\sigma)$ and a sequence of homomorphisms $\eta_{q}$ as follows. 
(Although this algorithm was actually given for Milnor invariants of links in $S^{3}$, it can be applied to those of string links.)

Given an $m$-component string link $\sigma$, 
consider its diagram $D_{1}\cup\cdots\cup D_{m}$. 
Put labels $a_{i1},a_{i2},\ldots,a_{ir(i)}$ in order on all arcs of the $i$th component $D_{i}$ 
while we go along orientation on $D_{i}$ from the initial arc, 
where $r(i)$ denotes the number of arcs of $D_{i}$ $(i=1,\ldots,m)$. 
Then the Wirtinger presentation of $G(\sigma)$ has the form 
\[
\left\langle a_{ij}\ (1\leq i\leq m,1\leq j\leq r(i))~\vline~a_{ij+1}^{-1}u_{ij}^{-1}a_{ij}u_{ij}\ (1\leq i\leq m,1\leq j\leq r(i)-1)\right\rangle, 
\]
where the $u_{ij}$ are generators or inverses of generators which depend on the signs of the crossings.  
Here we set 
\[
v_{ij}=u_{i1}u_{i2}\ldots u_{ij}. 
\]
Let $\overline{A}$ denote the free group on the Wirtinger generators $\{a_{ij}\}$, 
and let $A$ denote the free subgroup generated by $a_{11},a_{21},\ldots,a_{m1}$. 
A sequence of homomorphisms $\eta_{q}:\overline{A}\rightarrow A$ is defined inductively by 
\begin{eqnarray*}
\eta_{1}(a_{ij})=a_{i1},\ \eta_{q+1}(a_{i1})=a_{i1}, \\
\eta_{q+1}(a_{ij+1})=\eta_{q}(v_{ij}^{-1}a_{i1}v_{ij}). 
\end{eqnarray*}
Let $\overline{A}_{q}$ denote the $q$th term of the lower central series of $\overline{A}$, and 
let $N$ denote the normal subgroup of $\overline{A}$ generated by the Wirtinger relations $\{a_{ij+1}^{-1}u_{ij}^{-1}a_{ij}u_{ij}\}$. 
Milnor proved that 
\begin{equation}\label{eq-Milnor}
\eta_{q}(a_{ij})\equiv a_{ij}\pmod{\overline{A}_{q}N}.  
\end{equation}
By the construction of the Wirtinger presentation, 
$a_{i1}$ represents the $i$th meridian of $\sigma$. 
Hence, we have the natural homomorphism 
\[
\phi:A\longrightarrow\langle\alpha_{1},\ldots,\alpha_{m}\rangle 
\] 
defined by $\phi(a_{i1})=\alpha_{i}$ $(i=1,\ldots,m)$.
Since $v_{ir(i)-1}=u_{i1}\ldots u_{ir(i)-1}$ represents an $i$th longitude, 
for the preferred longitude $l_{i}$  
we regard that $l_{i}=a_{i1}^{s}v_{ir(i)-1}$ for some $s\in\mathbb{Z}$. 
Moreover, we can identify $\phi\circ\eta_{q}(l_{i})$ with $\lambda_{i}$ by Congruence~(\ref{eq-Milnor}).

\section{Milnor invariants and $2n$-moves} 
In this section, we discuss the invariance of Milnor invariants under $2n$-moves. 

\subsection{Milnor link-homotopy invariants and $2n$-moves}
The following theorem reveals how Milnor link-homotopy invariants, i.e. $\mu$-invariants for non-repeated sequences, behave under $2n$-moves.

\begin{theorem}\label{th-invariance}
Let $n$ be a positive integer. 
If two string links $\sigma$ and $\sigma'$ are $(2n+{\rm lh})$-equivalent, 
then $\mu_{\sigma}(I)\equiv\mu_{\sigma'}(I)\pmod{n}$ for any non-repeated sequence~$I$. 
\end{theorem}

For $P,Q\in\mathbb{Z}\langle\langle X_{1},\cdots,X_{m}\rangle\rangle$, we use the notation $P\overset{(n)}{\equiv}Q$ if $P-Q$ is contained in the ideal generated by $n$. 
To show Theorem~\ref{th-invariance} we need the following lemma.

\begin{lemma}\label{lem-invariance}
Let $n\geq2$ be an integer and $\sigma$ an $m$-component string link. 
For any Wirtinger generators $a_{ij}$ and $a_{kl}$ of $G(\sigma)$, 
there exists $R(X_{i},X_{k})\in\mathbb{Z}\langle\langle X_{1},\cdots,X_{m}\rangle\rangle$ such that 
 each term of $R(X_{i},X_{k})$ contains $X_{i}$ and $X_{k}$, and  
\[
E\left(\phi\circ\eta_{q}\left(\left(a_{ij}^{\e} a_{kl}^{\de}\right)^{\pm n}\right)\right)\overset{(n)}{\equiv}1+\binom{n}{2}R(X_{i},X_{k})+\mathcal{O}(2), 
\] 
where $\e,\de\in\{1,-1\}$ 
and $\mathcal{O}(2)$ denotes $0$ or the terms containing $X_{r}$ at least two for some $r$ $(=1,\ldots,m)$. 
\end{lemma}

\begin{proof}
By the definition of $\eta_{q}$, 
$\phi\circ\eta_{q}\left(a_{ij}^{\e}\right)=w^{-1}\alpha_{i}^{\e}w$ for some word $w$ in $\alpha_{1},\ldots,\alpha_{m}$. 
Set $E(w)=1+W$ and $E(w^{-1})=1+\overline{W}$, 
where $W$ and $\overline{W}$ denote the terms of degree $\geq1$ such that $\left(1+\overline{W}\right)\left(1+W\right)=1$. 
It follows that 
\begin{eqnarray*}
E\left(\phi\circ\eta_{q}\left(a_{ij}^{\e}\right)\right) 
&=&E\left(w^{-1}\alpha_{i}^{\e}w\right) \\ 
&=&\left(1+\overline{W}\right)\left(1+\e X_{i}\right)\left(1+W\right)+\mathcal{O}(2) \\
&=&1+\e X_{i}+\e X_{i}W+\e\overline{W}X_{i}+\e\overline{W}X_{i}W+\mathcal{O}(2) \\
&=&1+\e P(X_{i})+\mathcal{O}(2), 
\end{eqnarray*}
where $P(X_{i})=X_{i}+X_{i}W+\overline{W}X_{i}+\overline{W}X_{i}W$. 
Note that each term in $P(X_{i})$ contains $X_{i}$. 
Similarly, we have 
\[
E\left(\phi\circ\eta_{q}\left(a_{kl}^{\de}\right)\right)=1+\de Q(X_{k})+\mathcal{O}(2),
\]
where $Q(X_{k})$ denotes the terms of degree $\geq1$, each of which contains $X_{k}$. 
Therefore we have the following. 
\begin{eqnarray*}
&&E\left(\phi\circ\eta_{q}\left(\left(a_{ij}^{\e} a_{kl}^{\de}\right)^{n}\right)\right) \\
&&=\left(\left(1+\e P(X_{i})+\mathcal{O}(2)\right)\left(1+\de Q(X_{k})+\mathcal{O}(2)\right)\right)^{n} \\
&&=\left(1+\e P(X_{i})+\de Q(X_{k})+\e\de P(X_{i})Q(X_{k})+\mathcal{O}(2)\right)^{n} \\
&&=1+\sum_{r=1}^{n}\binom{n}{r}
\left(\e P(X_{i})+\de Q(X_{k})+\e\de P(X_{i})Q(X_{k})+\mathcal{O}(2)\right)^{r} \\
&&\overset{(n)}{\equiv}1+\binom{n}{2}
\left(P(X_{i})+Q(X_{k})+P(X_{i})Q(X_{k})+\mathcal{O}(2)\right)^{2}+\mathcal{O}(2) \\
&&=1+\binom{n}{2}
\left(P(X_{i})Q(X_{k})+Q(X_{k})P(X_{i})+\mathcal{O}(2)\right)+\mathcal{O}(2) \\
&&=1+\binom{n}{2}
\left(P(X_{i})Q(X_{k})+Q(X_{k})P(X_{i})\right)+\mathcal{O}(2). 
\end{eqnarray*}
Similarly, we have 
\begin{eqnarray*}
E\left(\phi\circ\eta_{q}\left(\left(a_{ij}^{\e} a_{kl}^{\de}\right)^{-n}\right)\right) 
&=&E\left(\phi\circ\eta_{q}\left(\left(a_{kl}^{-\de}a_{ij}^{-\e} \right)^{n}\right)\right) \\
&\overset{(n)}{\equiv}& 
1+\binom{n}{2}\left(Q(X_{k})P(X_{i})+P(X_{i})Q(X_{k})\right)
+\mathcal{O}(2) \\
&=&1+\binom{n}{2}
\left(P(X_{i})Q(X_{k})+Q(X_{k})P(X_{i})\right)+\mathcal{O}(2).   
\end{eqnarray*}
Setting $R(X_{i},X_{k})=P(X_{i})Q(X_{k})+Q(X_{k})P(X_{i})$, we obtain the desired congruence. 
\end{proof}

\begin{proof}[Proof of Theorem~$\ref{th-invariance}$] 
It is obvious for $n=1$, and hence we consider the case $n\geq2$. 
Since $\mu$-invariants for non-repeated sequences are link-homotopy invariants, 
we show that their residue classes modulo $n$ are preserved under $2n$-moves. 

Assume that two $m$-component string links $\sigma$ and $\sigma'$ are related by a single $2n$-move. 
A $2n$-move involving two strands of a single component is realized by link-homotopy. 
Furthermore, a $2n$-move whose two strands are oriented antiparallel is generated by link-homotopy and a $2n$-move whose strands are oriented parallel, see Figure~\ref{anti-parallel}. 
(Note that $2$-component string links having the same linking number are link-homotopic.)
Thus, we may assume that two strands performing the $2n$-move, which relates $\sigma$ to $\sigma'$, are oriented parallel and belong to different components.

\begin{figure}[htbp]
  \begin{center}
    \vspace{1em}
    \begin{overpic}[width=12cm]{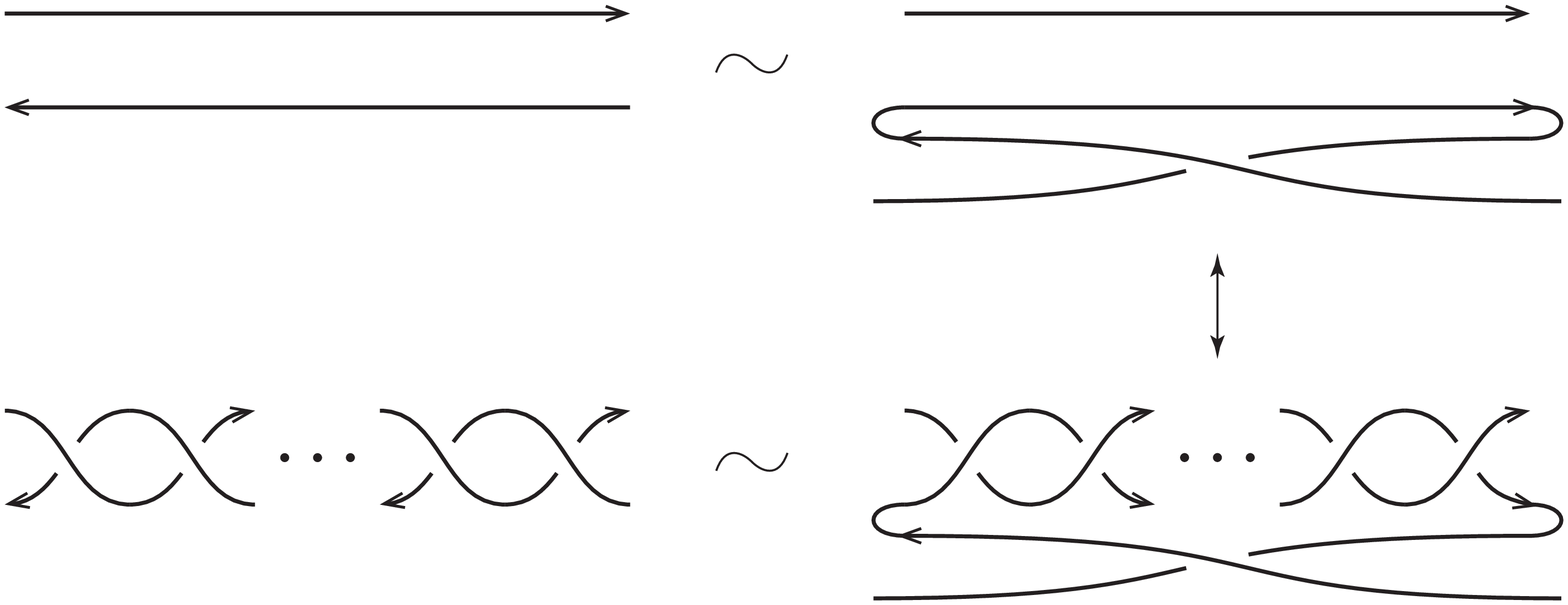}
      \put(151,124){{\footnotesize isotopy}}
      \put(141,37){{\footnotesize link-homotopy}}
      \put(271,61){$2n$-move}
      \put(12,40){{\footnotesize 1}}
      \put(118.5,40){{\footnotesize $2n$}}
      \put(208.5,40){{\footnotesize 1}}
      \put(315,40){{\footnotesize $2n$}}
    \end{overpic}
  \end{center}
  \caption{}
  \label{anti-parallel}
\end{figure}

There are diagrams $D$ and $D'$ of $\sigma$ and $\sigma'$, respectively, which are identical except in a disk $\Delta$ where they differ as illustrated in Figure~\ref{2n-move}. 
(It can be seen that  the move in the disk $\Delta$ of Figure~\ref{2n-move} is equivalent to a $2n$-move.) 
Put labels $a_{ij}$ $(1\leq i\leq m$, $1\leq j\leq r(i))$ on all arcs of $D$ as described in Section~\ref{subsec-algorithm}, 
and put labels $a'_{ij}$ on all arcs in $D'\setminus\Delta$ 
which correspond to the arcs labeled $a_{ij}$ in $D\setminus\Delta$. 
Also put labels $b'_{1},\ldots,b'_{2n},c'_{1},\ldots,c'_{2n}$ on the arcs of $D'$ in $\Delta$ as illustrated in Figure~\ref{2n-move}. 
Let $\overline{A'}$ be the free group on $\{a'_{ij}\}\cup\{b'_{1},\ldots,b'_{2n},c'_{1},\ldots,c'_{2n}\}$ 
and $A'$ the free subgroup on $\{a'_{11},a'_{21},\ldots,a'_{m1}\}$. 
Let $\eta'_{q}:\overline{A'}\rightarrow A'$ denote the sequence of homomorphisms associated with $D'$ given in Section~\ref{subsec-algorithm}, 
and define a homomorphism $\phi':A'\rightarrow\langle\alpha_{1},\ldots,\alpha_{m}\rangle$ by $\phi'(a'_{i1})=\alpha_{i}$ $(i=1,\ldots,m)$.

\begin{figure}[htbp]
  \begin{center}
    \vspace{1em}
    \begin{overpic}[width=11cm]{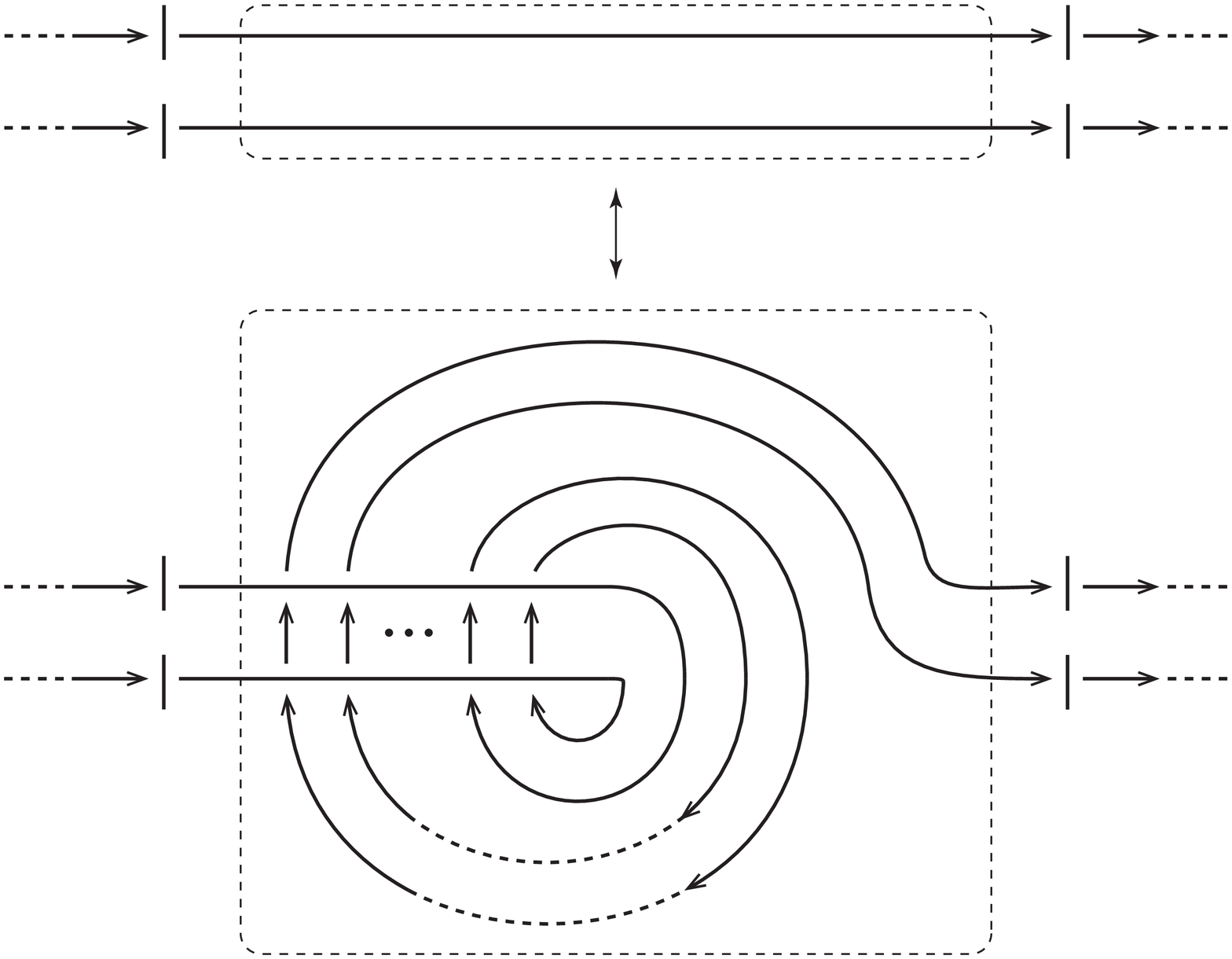}
      \put(-21,219){$D$ :}
      \put(-21,79){$D'$ :}
      \put(162,180){$2n$-move}
      \put(152,244){$\Delta$}
      \put(152,-10){$\Delta$}
      \put(11,239){$a_{kl-1}$}
      \put(11,201){$a_{gh-1}$}
      \put(47,239){$a_{kl}$}
      \put(47,201){$a_{gh}$}
      \put(277,239){$a_{kl+1}$}
      \put(277,201){$a_{gh+1}$}
      \put(11,101){$a'_{kl-1}$}
      \put(11,59){$a'_{gh-1}$}
      \put(46,101){$a'_{kl}$}
      \put(46,59){$a'_{gh}$}
      \put(277,101){$a'_{kl+1}$}
      \put(277,59){$a'_{gh+1}$}
      \put(138,78){{\footnotesize $b'_{1}$}}
      \put(122,78){{\footnotesize $c'_{1}$}}
      \put(155,99){{\footnotesize $b'_{2}$}}
      \put(137,123){{\footnotesize $c'_{2}$}}
      \put(98,110){{\footnotesize $b'_{2n}$}}
      \put(67,126){{\footnotesize $c'_{2n}$}}
      \put(102,42){{\footnotesize $b'_{2n-2}$}}
      \put(63,29){{\footnotesize $c'_{2n-2}$}}
    \end{overpic}
  \end{center}
  \caption{$D$ and $D'$ are related by a single $2n$-move.}
  \label{2n-move}
\end{figure}

For the $i$th preferred longitudes $l_{i}$ and $l'_{i}$ associated with $D$ and $D'$, respectively, 
it is enough to show that 
\begin{equation}\label{eq-longitude}
E\left(\phi\circ\eta_{q}\left(l_{i}\right)\right)
\overset{(n)}{\equiv}E\left(\phi'\circ\eta'_{q}\left(l'_{i}\right)\right)+\mathcal{O}(2)+\mathcal{P}(X_{i})
\end{equation}
for any $1\leq i\leq m$, 
where $\mathcal{P}(X_{i})$ denotes the terms containing $X_{i}$. 
To show the congruence above, we need the following claim.

\begin{claim}\label{eq-generator}
For any $1\leq i\leq m$ and $1\leq j\leq r(i)$, we have 
\begin{equation*}
E\left(\phi\circ\eta_{q}\left(a_{ij}\right)\right)
\overset{(n)}{\equiv}E\left(\phi'\circ\eta'_{q}\left(a'_{ij}\right)\right)+\mathcal{O}(2). 
\end{equation*}  
\end{claim}

Before showing Claim~\ref{eq-generator}, 
we observe that it implies Congruence~(\ref{eq-longitude}). 
Without loss of generality we may assume that $i=1$, 
i.e. we compare the preferred longitudes $l_{1}=a_{11}^{s}v_{1r(1)-1}$ and $l'_{1}={a'}^{t}_{11}v'_{1r(1)-1}$ $(s,t\in\mathbb{Z})$. 
Since two strands in $\Delta$ belong to different components, 
we only need to consider two cases. 

If both of the two strands in $\Delta$ do not belong to the $1$st component, then $s=t$ and 
$l'_{1}$ is obtained from $l_{1}$ by replacing $u_{1j}$ with $u'_{1j}$ $(j=1,\ldots,r(1)-1)$ and $a_{11}$ with $a'_{11}$. 
Therefore, Congruence~(\ref{eq-longitude}) follows from Claim~\ref{eq-generator}. 

If one of the two strands in $\Delta$ belongs to the $1$st component, then Figure~\ref{2n-move} indicates that 
$l_{1}$ and $l'_{1}$ can be written respectively in the forms 
\[
l_{1}=a_{11}^{s}u_{11}\ldots u_{1h-1}u_{1h}\ldots u_{1r(1)-1}
\]
and
\[
l'_{1}={a'}^{s-n}_{11}u'_{11}\ldots u'_{1h-1}\left(a'_{1h}a'_{kl}\right)^{n}u'_{1h}\ldots u'_{1r(1)-1}.
\]
Both $E\left(\phi\circ\eta_{q}\left(a_{11}^{s}\right)\right)$ and $E\left(\phi'\circ\eta'_{q}\left({a'}^{s-n}_{11}\right)\right)$ have the form 
$1+\mathcal{P}(X_{1})$. 
Furthermore, by Lemma~\ref{lem-invariance}, we have 
\[
E\left(\phi'\circ\eta'_{q}\left(\left(a'_{1h}a'_{kl}\right)^{n}\right)\right)
\overset{(n)}{\equiv}1+\binom{n}{2}R(X_{1},X_{k})+\mathcal{O}(2)
=1+\mathcal{P}(X_{1})+\mathcal{O}(2). 
\]
Therefore, this together with Claim~\ref{eq-generator} proves Congruence~(\ref{eq-longitude}).

Now, we turn to the proof of Claim~\ref{eq-generator}. 
The proof is done by induction on $q$. 
The assertion certainly holds for $q=1$. 
Recall that 
\[
\phi\circ\eta_{q+1}\left(a_{ij+1}\right)=\phi\circ\eta_{q}\left(v_{ij}^{-1}a_{i1}v_{ij}\right)\]
and
\[\phi'\circ\eta'_{q+1}\left(a'_{ij+1}\right)=\phi'\circ\eta'_{q}\left({v'_{ij}}^{-1}a'_{i1}v'_{ij}\right).
\]

If $v_{ij}$ does not pass through $\Delta$, 
then it is clear that $v'_{ij}$ is obtained from $v_{ij}$ by replacing $a_{ij}$ with $a'_{ij}$,  
and hence $E\left(\phi\circ\eta_{q}\left(v_{ij}\right)\right)\overset{(n)}{\equiv}E\left(\phi'\circ\eta'_{q}\left(v'_{ij}\right)\right)+\mathcal{O}(2)$ by the induction hypothesis. 
This implies that  
\begin{eqnarray*}
E\left(\phi\circ\eta_{q+1}\left(a_{ij+1}\right)\right)
&=&E\left(\phi\circ\eta_{q}\left(v_{ij}^{-1}a_{i1}v_{ij}\right)\right) \\
&\overset{(n)}{\equiv}&E\left(\phi'\circ\eta'_{q}\left({v'_{ij}}^{-1}a'_{i1}v'_{ij}\right)\right)+\mathcal{O}(2) \\
&=&E\left(\phi'\circ\eta'_{q+1}\left(a'_{ij+1}\right)\right)+\mathcal{O}(2). 
\end{eqnarray*}

If $v_{ij}$ passes through $\Delta$, 
then $v_{ij}$ and $v'_{ij}$ can be written respectively in the forms 
\[
v_{ij}=u_{i1}\ldots u_{ih-1}u_{ih}\ldots u_{ij}
\]
and
\[ 
v'_{ij}=u'_{i1}\ldots u'_{ih-1}\left(a'_{ih}a'_{kl}\right)^{n}u'_{ih}\ldots u'_{ij}.
\]  
Set $E\left(\phi\circ\eta_{q}\left(u_{i1}\ldots u_{ih-1}\right)\right)=1+F$, 
$E\left(\phi\circ\eta_{q}\left(\left(u_{i1}\ldots u_{ih-1}\right)^{-1}\right)\right)=1+\overline{F}$, 
$E\left(\phi\circ\eta_{q}\left(u_{ih}\ldots u_{ij}\right)\right)=1+G$ 
and $E\left(\phi\circ\eta_{q}\left(\left(u_{ih}\ldots u_{ij}\right)^{-1}\right)\right)=1+\overline{G}$, 
where $F,\overline{F},G$ and $\overline{G}$ denote the terms of degree $\geq1$. 
Then we have   
\begin{equation*}
E\left(\phi\circ\eta_{q+1}\left(a_{ij+1}\right)\right)
=\left(1+\overline{G}\right)\left(1+\overline{F}\right)\left(1+X_{i}\right)\left(1+{F}\right)\left(1+G\right). 
\end{equation*}
It follows from the induction hypothesis that 
\begin{eqnarray*}
E\left(\phi'\circ\eta'_{q+1}\left(a'_{ij+1}\right)\right) 
&\overset{(n)}{\equiv}&\left(1+\overline{G}\right)E\left(\phi'\circ\eta'_{q}\left(\left(a'_{ih}a'_{kl}\right)^{-n}\right)\right)
\left(1+\overline{F}\right)\left(1+X_{i}\right) \\
&&\times\left(1+F\right)
E\left(\phi'\circ\eta'_{q}\left(\left(a'_{ih}a'_{kl}\right)^{n}\right)\right)\left(1+G\right)+\mathcal{O}(2). 
\end{eqnarray*}
Lemma~\ref{lem-invariance} implies that 
\begin{eqnarray*}
E\left(\phi'\circ\eta'_{q+1}\left(a'_{ij+1}\right)\right) 
&\overset{(n)}{\equiv}&\left(1+\overline{G}\right)
\left(1+\binom{n}{2} R(X_{i},X_{k})\right)
\left(1+\overline{F}\right)\left(1+X_{i}\right) \\
&&\times\left(1+F\right)
\left(1+\binom{n}{2} R(X_{i},X_{k})\right)
\left(1+G\right)+\mathcal{O}(2). 
\end{eqnarray*}
In particular, we have the following. 
\begin{eqnarray*}
&&\left(1+\binom{n}{2} R(X_{i},X_{k})\right)
\left(1+\overline{F}\right)\left(1+X_{i}\right)
\left(1+F\right)\left(1+\binom{n}{2} R(X_{i},X_{k})\right) \\
&&=\left(1+\binom{n}{2} R(X_{i},X_{k})\right)
\left(1+\left(1+\overline{F}\right)X_{i}\left(1+F\right)\right)\left(1+\binom{n}{2} R(X_{i},X_{k})\right) \\
&&=1+\left(1+\overline{F}\right)X_{i}\left(1+F\right)+2\binom{n}{2} R(X_{i},X_{k})+\mathcal{O}(2) \\
&&\overset{(n)}{\equiv}
1+\left(1+\overline{F}\right)X_{i}\left(1+F\right)+\mathcal{O}(2) \\
&&=\left(1+\overline{F}\right)\left(1+X_{i}\right)\left(1+{F}\right)+\mathcal{O}(2). 
\end{eqnarray*}
This proves Claim~\ref{eq-generator}, and hence completes the proof of Theorem~\ref{th-invariance}. 
\end{proof}

\subsection{Milnor isotopy invariants and $2p$-moves}

For Milnor isotopy invariants, i.e. $\mu$-invariants possibly with {\em repeated} sequences, we have the following.

\begin{proposition}\label{prop-inv-prime}
Let $p$ be a prime number. 
If two string links $\sigma$ and $\sigma'$ are $2p$-move equivalent, then $\mu_{\sigma}(I)\equiv\mu_{\sigma'}(I)\pmod{p}$ for any sequence $I$ of length $\leq p$. 
\end{proposition}

\begin{remark}
(1)~
The restriction on the length of sequences in Proposition~\ref{prop-inv-prime} must be necessary. 
In fact, there exists the following example: 
Let $\sigma=\sigma_{1}^{4}$, where $\sigma_{1}$ is the generator of $2$-braids. 
We can verify that $\mu_{\sigma}(112)=1$ by using a computer program written by Takabatake, Kuboyama and Sakamoto~\cite{TKS}.\footnote{Using the technique of ``grammar compression'', Takabatake, Kuboyama and Sakamoto~\cite{TKS} made a computer program in the program language C++, based on Milnor's algorithm, which is able to give us $\mu$-invariants of length at least $\leq16$.} 
While $\sigma$ is $4$-move equivalent to $\mathbf{1}_{2}$, $\mu_{\sigma}(112)$ is not congruent to $0$ modulo~$2$. 

\noindent
(2)~
Proposition~\ref{prop-inv-prime} cannot be extended to the $2n$-move equivalence classes of string links for a nonprime number $n$. 
For example, let $\sigma=\sigma^{8}_{1}$ then  
the computer program of Takabatake-Kuboyama-Sakamoto gives us that $\mu_{\sigma}(211)=10$. 
While $\sigma$ is $8$-move equivalent to $\mathbf{1}_{2}$, $\mu_{\sigma}(211)$ is not congruent to $0$ modulo~$4$. 
\end{remark}

\begin{proof}[Proof of Proposition~$\ref{prop-inv-prime}$]
Let $D$ and $D'$ be diagrams of $m$-component string links $\sigma$ and $\sigma'$, respectively. 
Assume that $D$ and $D'$ are related by a single $2p$-move whose strands are oriented parallel. 
(In the case where the orientations of two strands of a $2p$-move are antiparallel, the proof is strictly similar. 
Hence, we omit the case.) 
We use the same notation as in the proof of Theorem~\ref{th-invariance}.  
It is enough to show that, for any $1\leq i\leq m$,  
\begin{equation*}
E\left(\phi\circ\eta_{q}\left(l_{i}\right)\right)
\overset{(p)}{\equiv}E\left(\phi'\circ\eta'_{q}\left(l'_{i}\right)\right)+(\text{terms of degree $\geq p$}).  
\end{equation*} 
By arguments similar to those in the proof of Theorem~\ref{th-invariance}, 
$l'_{i}$ is obtained from $l_{i}$ by replacing $a_{kl}$ with $a'_{kl}$ for some $k,l$ and inserting the $p$th powers of elements in the free group $\overline{A'}$ on the Wirtinger generators of $G(\sigma')$.  
The following claim completes the proof. 
\end{proof}

\begin{claim}\label{claim-prime} 
(1)~
For any word $w$ in $\alpha_{1},\ldots,\alpha_{m}$, we have 
\[
E\left(w^{p}\right)
\overset{(p)}{\equiv}
1+(\text{terms of degree $\geq p$}). 
\]
(2)~
For any $1\leq i\leq m$ and $1\leq j\leq r(i)$, we have 
\[
E\left(\phi\circ\eta_{q}\left(a_{ij}\right)\right)
\overset{(p)}{\equiv}
E\left(\phi'\circ\eta'_{q}\left({a'}_{ij}\right)\right)
\overset{(p)}{\equiv}
1+(\text{terms of degree $\geq p$}).
\] 
\end{claim}

\begin{proof}
Set $E\left(w\right)=1+W$, where $W$ denotes the terms of degree~$\geq1$. 
Then $E\left(w^{p}\right)=\left(1+W\right)^{p}\overset{(p)}{\equiv}1+W^{p}$.
This proves Claim~\ref{claim-prime}~(1). 

By arguments similar to those in the proof of Claim~\ref{eq-generator}, 
$\eta'_{q+1}\left(a'_{ij}\right)$ is obtained from $\eta_{q+1}\left(a_{ij}\right)$ by replacing $\eta_{q}\left(a_{kl}\right)$ with $\eta'_{q}\left(a'_{kl}\right)$ some $k,l$ and 
inserting $\eta'_{q}\left(w^{p}\right)$ for 
some elements $w$ in $\overline{A'}$. 
Therefore, using Claim~\ref{claim-prime}~(1), 
we complete the proof of Claim~\ref{claim-prime}~(2) by induction on $q$. 
\end{proof}

\section{Claspers}\label{sec-clasper} 
To show Theorem~\ref{th-sl}, we will use the theory of claspers introduced by K.~Habiro in~\cite{H}. 
In this section, we briefly recall the basic notions of clasper theory from~\cite{H}.
We only need the notion of $C_{k}$-tree in this paper, 
and refer the reader to~\cite{H} for the general definition of claspers.

\subsection{Definitions}
\begin{definition}
Let $\sigma$ be a string link in $\mathbb{D}^{2}\times [0,1]$. 
An embedded disk $T$ in $\mathbb{D}^{2}\times [0,1]$ is called a {\it tree clasper} for $\sigma$ if it satisfies the following:
\begin{enumerate}
\item $T$ decomposes into disks and bands. 

\item Bands are called {\em edges} and each of them connects two distinct disks.

\item Each disk has either one or three incident edges, 
and is then respectively called a {\em disk-leaf} or {\em node}.
\item $\sigma$ intersects $T$ transversely and the intersections are contained in the union of the interior of the disk-leaves.
\end{enumerate}
We say that $T$ is a {\em $C_{k}$-tree} if the number of disk-leaves of $T$ is $k+1$, and is {\it simple} if each disk-leaf of $T$ intersects $\sigma$ at a single point.
(Note that a tree clasper is called a {\it strict tree clasper} in~\cite{H}.)
\end{definition}

We will make use of the drawing convention for claspers of~\cite[Figure~$7$]{H} except for the following: 
a \textcircled{$+$} (resp. \textcircled{$-$}) on an edge represents 
a positive (resp. negative) half-twist.
This replaces the circled $S$ (resp. $S^{-1}$) notation
used in~\cite{H}.

Given a $C_k$-tree $T$ for a string link $\sigma$, there is a procedure to construct a zero-framed link $\gamma(T)$ 
in the complement of $\sigma$.
{\it Surgery along $T$} means surgery along~$\gamma(T)$.
Since surgery along $\gamma(T)$ preserves the ambient space, surgery along the $C_k$-tree $T$ can be regarded as a local move on $\sigma$ in $\mathbb{D}^{2}\times [0,1]$.
Denote by $\sigma_{T}$ the string link in $\mathbb{D}^{2}\times [0,1]$ which is obtained from $\sigma$ by surgery along $T$.
Similarly, 
we define the string link $\sigma_{T_1\cup\cdots \cup T_r}$ obtained from $\sigma$ by surgery along a disjoint union of tree claspers $T_1\cup\cdots \cup T_r$.
A $C_k$-tree $T$ having the shape of the tree clasper in Figure~\ref{linear} (with possibly some half-twists on the edges of $T$)  
is called a {\em linear $C_k$}-tree. 
As illustrated in Figure~\ref{linear}, surgery along a simple linear $C_{k}$-tree for $\sigma$ is ambient isotopic to a band summing of $\sigma$ and the $(k+1)$-component Milnor link\footnote{Also referred to as the {\em Sutton Hoo link} because of a cauldron chain from the Sutton Hoo exhibited in the British Museum~\cite[page 222]{F}.}
~(see~\cite[Fig. 7]{M54}).

\begin{figure}[htbp]
  \begin{center}
    \begin{overpic}[width=12cm]{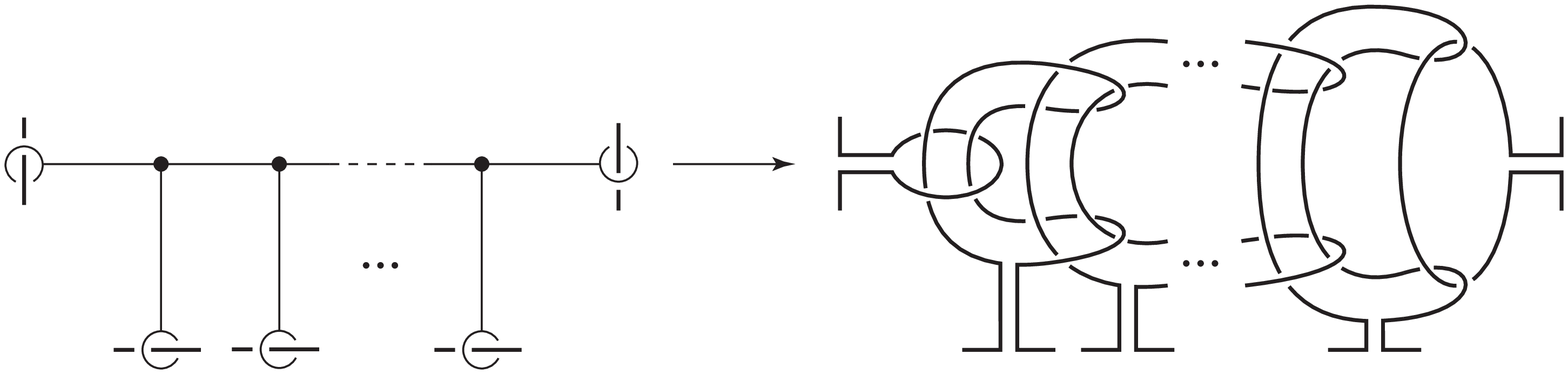}
      \put(143,52){surgery}
    \end{overpic}
  \caption{Surgery along a simple linear $C_{k}$-tree}
  \label{linear}
  \end{center}
\end{figure}

The {\em $C_{k}$-equivalence} is the equivalence relation on string links generated by surgery along $C_{k}$-trees and ambient isotopies. 
Habiro proved that two string links $\sigma$ and $\sigma'$ are $C_{k}$-equivalent if and only if there exists a disjoint union of simple $C_{k}$-trees $T_{1}\cup\cdots\cup T_{r}$ such that $\sigma'$ is ambient isotopic to $\sigma_{T_{1}\cup\cdots\cup T_{r}}$~\cite[Theorem 3.17]{H}.
This implies that surgery along any $C_{k}$-tree can be replaced with surgery along a disjoint union of {\em simple} $C_{k}$-trees. 
Hereafter, by a $C_{k}$-tree we mean a simple $C_{k}$-tree.

\subsection{Some technical lemmas}
This subsection gives some lemmas, which will be used to show Theorem~\ref{th-sl}. 

Given a $C_k$-tree $T$ for an $m$-component string link $\sigma=\sigma_{1}\cup\cdots\cup \sigma_{m}$, 
the set $\{\ i~\vline~\sigma_{i}\cap T \neq \emptyset, 1\leq i\leq m\}$ is called the {\em index} of $T$ and is denoted by ${\rm Ind}(T)$. 
The following is a direct consequence of \cite[Lemma 1.2]{FY}.

\begin{lemma}[{cf.~\cite[Lemma 1.2]{FY}}]\label{lem-self}
Let $T$ be a $C_{k}$-tree for a string link $\sigma$ with $|{\rm Ind}(T)|\leq k$. 
Then $\sigma_{T}$ is link-homotopic to $\sigma$. 
\end{lemma}

The set of ambient isotopy classes of $m$-component string links has a monoid structure under the {\em stacking product} ``$*$'', and with the trivial $m$-component string link $\mathbf{1}_{m}$ as the unit element. 
Combining Lemma~\ref{lem-self} and \cite[Lemma 2.4]{Ytrans}, we have the following.

\begin{lemma}[{cf.~\cite[Lemma 2.4]{Ytrans}}]\label{lem-halftwist}
Let $T$ be a $C_{k}$-tree for $\mathbf{1}_{m}$, and let $\overline{T}$ be a $C_{k}$-tree obtained from $T$ by adding a half-twist on an edge. 
Then $(\mathbf{1}_{m})_{T}*(\mathbf{1}_{m})_{\overline{T}}$ is link-homotopic to $\mathbf{1}_{m}$. 
\end{lemma}

By Lemma~\ref{lem-self} together with~\cite[Lemma 2.2 (2) and Remark 2.3]{MY12}, we have the following.

\begin{lemma}[{cf.~\cite[Lemma 2.2 (2) and Remark 2.3]{MY12}}]\label{lem-cc}
Let $T_{1}$ be a $C_{k}$-tree for a string link $\sigma$, 
and $T_{2}$ a $C_{l}$-tree for $\sigma$. 
Let $T'_{1}\cup T'_{2}$ be obtained from $T_{1}\cup T_{2}$ by changing a crossing of an edge of $T_{1}$ and that of $T_{2}$. 
Then $\sigma_{T_{1}\cup T_{2}}$ is link-homotopic to $\sigma_{T'_{1}\cup T'_{2}}$. 
\end{lemma}

Here, by {\em parallel} tree claspers we mean a family of $r$ parallel copies of a tree clasper $T$ for some $r\geq1$. 
We call $r$ the {\em multiplicity} of the parallel clasper. 
The following can be proved by Lemma~\ref{lem-self} and ~\cite[Lemma 2.2 (1) and Remark 2.3]{MY12}.

\begin{lemma}[{cf.~\cite[Lemma 2.2 (1) and Remark 2.3]{MY12}}]\label{lem-sliding}
Let $T_{1}$ be a $C_{k}$-tree for a string link $\sigma$, 
and $T_{2}$ a parallel $C_{l}$-tree with multiplicity $r$ for $\sigma$. 
Let $T'_{1}\cup T'_{2}$ be obtained from $T_{1}\cup T_{2}$ by sliding a leaf $f$ of $T_{1}$ over $r$ parallel leaves of $T_{2}$ $($see Figure~$\ref{sliding}$$)$. 
Then $\sigma_{T_{1}\cup T_{2}}$ is link-homotopic to $\sigma_{T'_{1}\cup T'_{2}\cup Y}$, 
where $Y$ denotes the parallel $C_{k+l}$-tree with multiplicity $r$ obtained by inserting a vertex $v$ in the edge $e$ of $T_{2}$ and connecting $v$ to the edge incident to $f$ as illustrated in Figure~$\ref{sliding}$. 
\end{lemma}

\begin{figure}[htbp]
  \begin{center}
    \vspace{1em}
    \begin{overpic}[width=10cm]{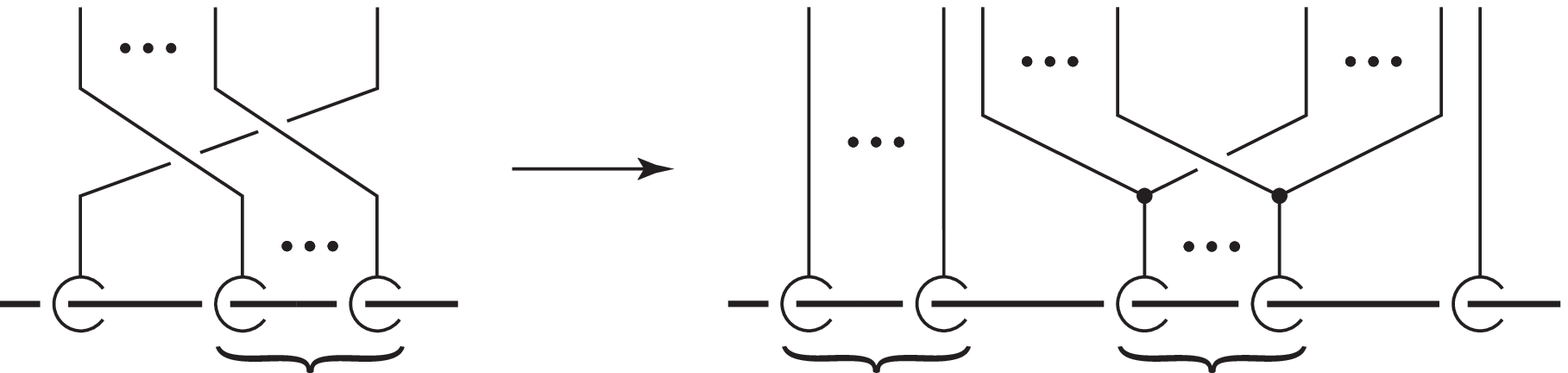}
      \put(23,71){$T_{2}$}
      \put(64,71){$T_{1}$}
      \put(-8,11){$\sigma$}
      \put(54.5,-8){$r$}
      \put(10,-3){$f$}
      \put(73,22){$e$}
      \put(217,71){$Y$}
      \put(237,27){$v$}
      \put(155,71){$T'_{2}$}
      \put(266,71){$T'_{1}$}
      \put(287,11){$\sigma$}
      \put(158,-8){$r$}
      \put(219,-8){$r$}
    \end{overpic}
  \vspace{1em}
  \caption{}
  \label{sliding}
  \end{center}
\end{figure}

\section{Proof of Theorem~\ref{th-sl}}
This section is devoted to the proof of Theorem~\ref{th-sl}. 

Habegger and Lin~\cite{HL} proved that Milnor link-homotopy invariants  classify string links up to link-homotopy. 
In~\cite{Ytrans}, the third author gave an alternative proof for this by using clasper theory. 
Actually, he constructed explicit representatives, determined by Milnor link-homotopy invariants, for the link-homotopy classes as follows. 
Let $\pi:\{1,\ldots,k\}\rightarrow\{1,\ldots,m\}$ $(2\leq k\leq m)$ be an injection such that $\pi(i)<\pi(k-1)<\pi(k)$ $(i=1,\ldots,k-2)$, 
and let $\mathcal{F}_{k}$ be the set of such injections. 
Given $\pi\in\mathcal{F}_{k}$, 
let $T_{\pi}$ and $\overline{T}_{\pi}$ be linear $C_{k-1}$-trees with index $\{\pi(1),\ldots,\pi(k)\}$ illustrated in the left- and right-hand side of Figure~\ref{representative}, respectively. 
Here, Figure~\ref{representative} describes the images of homeomorphisms from neighborhood of $T_{\pi}$ and $\overline{T}_{\pi}$ to the $3$-ball. 
Setting $V_{\pi}=(\mathbf{1}_{m})_{T_{\pi}}$ and $V^{-1}_{\pi}=(\mathbf{1}_{m})_{\overline{T}_{\pi}}$, we have the following.

\begin{figure}[htbp]
  \begin{center}
    \vspace{1em}
    \begin{overpic}[width=11cm]{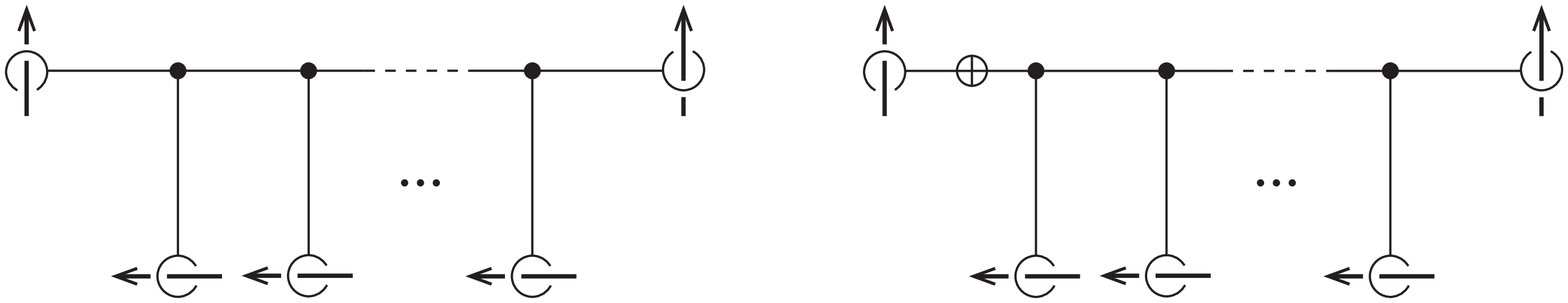}
      \put(70,54){$T_{\pi}$}
      \put(241,54){$\overline{T}_{\pi}$}
      \put(-3,63){$\pi(k)$}
      \put(118,63){$\pi(k-1)$}
      \put(12,-7){$\pi(1)$}
      \put(38,-7){$\pi(2)$}
      \put(111,-7){$\pi(k-2)$}
      \put(169,63){$\pi(k)$}
      \put(290,63){$\pi(k-1)$}
      \put(185,-7){$\pi(1)$}
      \put(210,-7){$\pi(2)$}
      \put(282,-7){$\pi(k-2)$}
    \end{overpic}
  \vspace{1em}
  \caption{Linear $C_{k-1}$-trees $T_{\pi}$ and $\overline{T}_{\pi}$ with index $\{\pi(1),\ldots,\pi(k)\}$}
  \label{representative}
  \end{center}
\end{figure}

\begin{theorem}[{\cite[Theorem~4.3]{Ytrans}}]\label{th-representative}
Let $\sigma$ be an $m$-component string link. 
Then $\sigma$ is link-homotopic to $\sigma_{1}*\cdots*\sigma_{m-1}$, where for each $k$,  
\[
\sigma_{k}=\prod_{\pi\in\mathcal{F}_{k+1}}V_{\pi}^{x_{\pi}}, 
\]
\[x_{\pi}=\left\{
\begin{array}{lll}
\mu_{\sigma}(\pi(1)\pi(2)) & (k=1), \\
\mu_{\sigma}(\pi(1)\ldots\pi(k+1))
-\mu_{\sigma_{1}*\cdots*\sigma_{k-1}}(\pi(1)\ldots\pi(k+1))
& (k\geq 2). 
\end{array}
\right.
\]
\end{theorem}

\vspace{0.5em}
The following is the key lemma to show Theorem~\ref{th-sl}.

\begin{lemma}\label{lem-del-clasper}
Let $n$ be a positive integer and $\e\in\{1,-1\}$.  
Then, for any $\pi\in\mathcal{F}_{k+1}$ $(1\leq k\leq m-1)$, $V_{\pi}^{\e n}$ is $(2n+{\rm lh})$-equivalent to $\mathbf{1}_{m}$. 
\end{lemma}

\begin{proof}
Since $V_{\pi}^{-n}*V_{\pi}^{n}$ is link-homotopic to $\mathbf{1}_{m}$ by Lemma~\ref{lem-halftwist}, 
it is enough to show the case $\e=1$, 
i.e. for any $\pi\in\mathcal{F}_{k+1}$, $V_{\pi}^{n}$ is $(2n+{\rm lh})$-equivalent to $\mathbf{1}_{m}$. 
For the case $k=1$, we see that $V_{\pi}^{n}$ and $\mathbf{1}_{m}$ are related by a single $2n$-move. 

Assume that $k\geq2$. 
Let $T_{1}$ be the linear $C_{k-1}$-tree  for $\mathbf{1}_{m}$ of Figure~\ref{del-clasper}~(a) with index $\{\pi(1),\ldots,\pi(k)\}$, 
and let $\overline{T}_{1}$ be obtained from $T_{1}$ by adding a positive half-twist on an edge. 
Then $\mathbf{1}_{m}$ is link-homotopic to $(\mathbf{1}_{m})_{\overline{T}_{1}\cup T_{1}}$ by Lemma~\ref{lem-halftwist}. 
Let $T_{2}$ be the parallel $C_{1}$-tree of Figure~\ref{del-clasper}~(b) with multiplicity $n$. 
Since surgery along $T_{2}$ is realized by a $2n$-move, 
$(\mathbf{1}_{m})_{\overline{T}_{1}\cup T_{1}}$ is $2n$-move equivalent to $(\mathbf{1}_{m})_{\overline{T}_{1}\cup T_{1}\cup T_{2}}$ in Figure~\ref{del-clasper}~(b). 
Let $T'_{1}\cup T'_{2}$ be obtained from $T_{1}\cup T_{2}$ by 
sliding a leaf of $T_{1}$ over $n$ parallel leaves of $T_{2}$, 
and let $Y$ be the parallel $C_{k}$-tree with multiplicity~$n$ as illustrated in Figure~\ref{del-clasper}~(c). 
It follows from Lemmas~\ref{lem-cc} and~\ref{lem-sliding} that 
$(\mathbf{1}_{m})_{\overline{T}_{1}\cup T_{1}\cup T_{2}}$ is link-homotopic to $(\mathbf{1}_{m})_{\overline{T}_{1}\cup T'_{1}\cup T'_{2}\cup Y}$. 
Furthermore, by Lemma~\ref{lem-halftwist}, 
$(\mathbf{1}_{m})_{\overline{T}_{1}\cup T'_{1}\cup T'_{2}\cup Y}$ is $(2n+{\rm lh})$-equivalent to $(\mathbf{1}_{m})_{Y}=V_{\pi}^{n}$. 
This completes the proof. 
\end{proof}

\begin{figure}[htbp]
  \begin{center}
    \vspace{1em}
    \begin{overpic}[width=7cm]{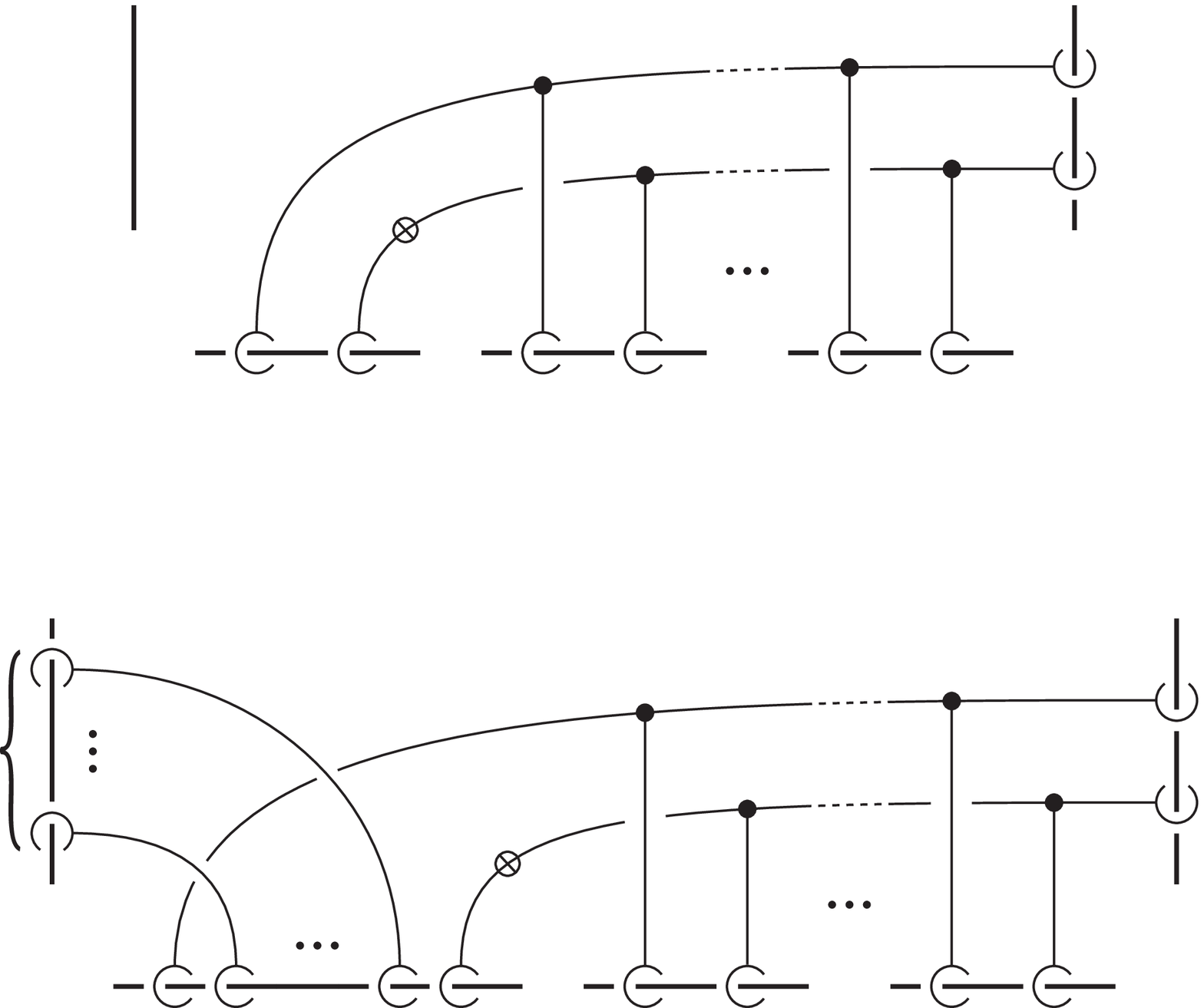}
      \put(94,85){(a)}
      \put(94,-15){(b)}
      \put(6,171){{\small $\pi(k+1)$}}
      \put(171,171){{\small $\pi(k)$}}
      \put(22,99){{\small $\pi(1)$}}
      \put(70,99){{\small $\pi(2)$}}
      \put(163,99){{\small $\pi(k-1)$}}
      \put(100,159){$T_{1}$}
      \put(111,143){$\overline{T}_{1}$}
      \put(116,53){$T_{1}$}
      \put(127,38){$\overline{T}_{1}$}
      \put(36,53){$T_{2}$}
      \put(-9,40.5){$n$}
    \end{overpic}
    
    \vspace{4em}
    \begin{overpic}[width=10cm]{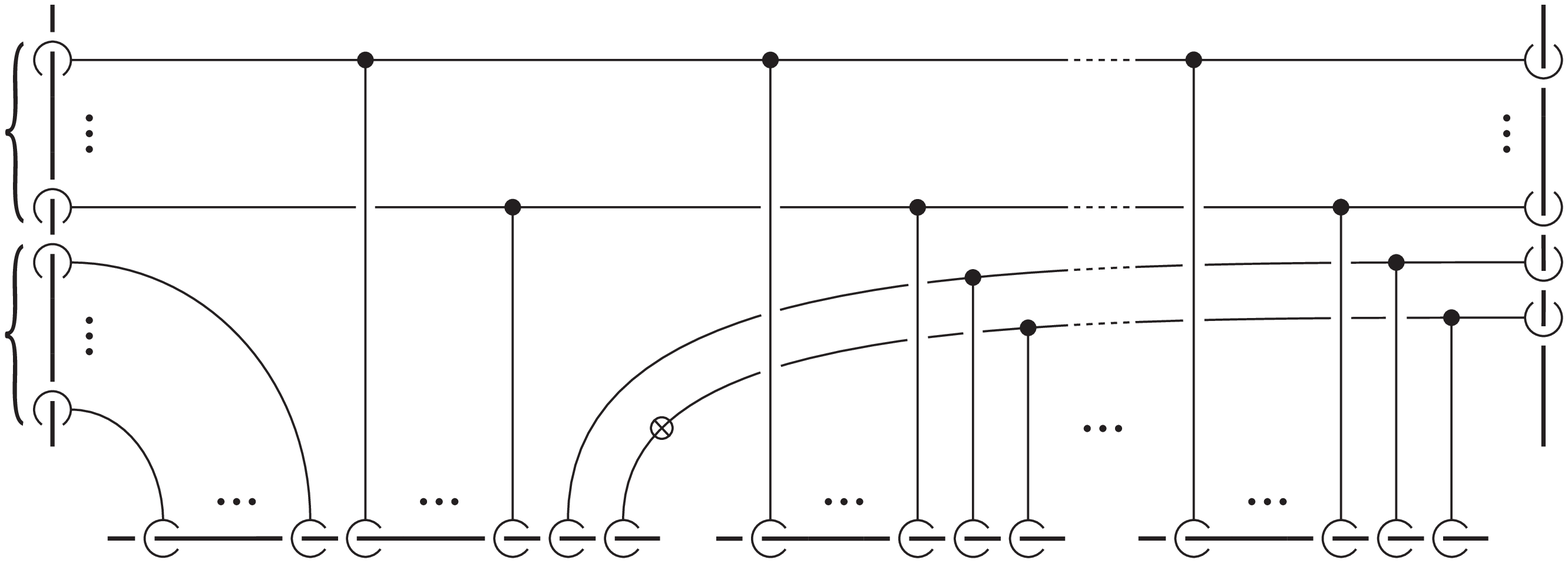}
      \put(140,-15){(c)}
      \put(165,93){$Y$}
      \put(112,41){$T'_{1}$}
      \put(125,18){$\overline{T}_{1}$}
      \put(45,41){$T'_{2}$}
      \put(-9,75){$n$}
      \put(-9,39){$n$}
    \end{overpic}
  \vspace{1em}
  \caption{}
  \label{del-clasper}
  \end{center}
\end{figure}

Combining Theorem~\ref{th-representative} and Lemma~\ref{lem-del-clasper}, 
we give a complete list of representatives for string links up to $(2n+{\rm lh})$-equivalence as follows.

\begin{proposition}\label{prop-rep-2nlh}
Let $\sigma$ be an $m$-component string link 
and $x_{\pi}$ as in Theorem~$\ref{th-representative}$. 
Then $\sigma$ is $(2n+{\rm lh})$-equivalent to $\tau_{1}*\cdots*\tau_{m-1}$, where for each $k$, 
\[
\tau_{k}=\prod_{\pi\in\mathcal{F}_{k+1}}V_{\pi}^{y_{\pi}} 
\]
with $0\leq y_{\pi}<n$ and $y_{\pi}\equiv x_{\pi}\pmod{n}$. 
\end{proposition}

\begin{proof}
It follows from Theorem~\ref{th-representative} that $\sigma$ is link-homotopic to $\sigma_{1}*\cdots*\sigma_{m-1}$, 
where
\[
\sigma_{k}=\prod_{\pi\in\mathcal{F}_{k+1}}V_{\pi}^{x_{\pi}}.
\] 
By Lemmas~\ref{lem-del-clasper} and~\ref{lem-halftwist}, 
we can insert/delete $V_{\pi}^{\pm n}$ and remove $V_{\pi}^{\e}*V_{\pi}^{-\e}$ up to $(2n+{\rm lh})$-equivalence $(\e\in\{1,-1\})$. 
Therefore, 
$\sigma_{k}$ is $(2n+{\rm lh})$-equivalent to $\tau_{k}$ for each $k$. 
\end{proof}

\begin{proof}[Proof of Theorem~$\ref{th-sl}$]
This follows from Theorem~\ref{th-invariance} and Proposition~\ref{prop-rep-2nlh}. 
\end{proof}

\begin{proof}[Proof of Corollary~$\ref{cor-group}$]
By combining Theorem~\ref{th-sl}, Lemma~\ref{lem-del-clasper} and Proposition~\ref{prop-rep-2nlh},
we have the corollary. 
\end{proof}

\begin{remark}
Theorem~\ref{th-sl} characterizes Milnor link-homotopy invariants modulo~$n$ by two local moves, the $2n$-move and self-crossing change. 
In~\cite{ABMW}, B.~Audoux, P.~Bellingeri, J.-B.~Meilhan and E.~Wagner defined Milnor invariants, denoted by~$\mu^{{\rm w}}$, for {\em welded string links} 
and proved that $\mu^{{\rm w}}$-invariants for non-repeated sequences classify welded string links up to {\em self-crossing virtualization}. 
(Later, this classification led to a link-homotopy classification of {\em $2$-dimensional string links} in $4$-space~\cite{AMW}). 
For welded string links, we can show a similar result to Theorem~\ref{th-sl} that characterizes $\mu^{{\rm w}}$-invariants for non-repeated sequences modulo $n$ in terms of the $2n$-move and self-crossing virtualization. 
While the idea of the proof is similar to that of Theorem~\ref{th-sl}, we need {\em arrow calculus} and representatives for welded string links up to self-crossing virtualization given in~\cite{MY19} 
instead of clasper calculus and representatives for string links up to link-homotopy. 
We will give the details in a future paper. 
\end{remark}

\section{Links in $S^{3}$}\label{sec-link}
In the previous sections, we have studied {\em string links}. 
We now address the case of {\em links} in $S^{3}$. 

Given an $m$-component string link $\sigma$, its {\em closure} is an $m$-component link in $S^{3}$ obtained from $\sigma$ by identifying points on the boundary of $\mathbb{D}^{2}\times [0,1]$ with their images under the projection $\mathbb{D}^{2}\times [0,1]\rightarrow \mathbb{D}^{2}$. 
The link inherits an ordering and orientation from $\sigma$. 
Note that every link can be represented by the closure of some string link. 

Habegger and Lin proved that for two link-homotopic links $L$ and $L'$, and for a string link $\sigma$ whose closure is $L$, there exists a string link $\sigma'$ whose closure is $L'$ such that $\sigma'$ is link-homotopic to $\sigma$~\cite[Lemma 2.5]{HL}.  
Similarly we have the following.

\begin{lemma}\label{lem-2nlh}
Let $n$ be a positive integer. 
Let $L$ and $L'$ be $(2n+{\rm lh})$-equivalent $($resp. $2n$-move equivalent$)$ 
links and $\sigma$ a string link whose closure is $L$. 
Then there exists a string link $\sigma'$ whose closure is $L'$ such that $\sigma'$ is $(2n+{\rm lh})$-equivalent $($resp. $2n$-move equivalent$)$ to $\sigma$. 
\end{lemma} 

\noindent
The proof is strictly similar to that of~\cite[Lemma 2.5]{HL}, and hence we omit it. 

Let $\sigma$ be a string link. 
We define $\Delta_{\sigma}(I)$ to be the greatest common divisor of all $\mu_{\sigma}(J)$ such that $J$ is obtained from $I$ by removing at least one index and permuting the remaining indices cyclically. 
It is known in~\cite{HL} that the integer $\Delta_{\sigma}(I)$ and the residue class of $\mu_{\sigma}(I)$ modulo ${\Delta_{\sigma}(I)}$ are invariants of the closure of $\sigma$.  
For a link $L$, 
we define $\Delta_{L}^{(n)}(I)$ to be $\gcd\{\Delta_{\sigma}(I),n\}$ and 
$\omu_{L}^{(n)}(I)$ to be the residue class of $\mu_{\sigma}(I)$ modulo $\Delta_{L}^{(n)}(I)$ 
for a string link $\sigma$ whose closure is $L$. 
Obviously, $\Delta_{L}^{(n)}(I)$ and $\omu_{L}^{(n)}(I)$ are invariants of $L$. 
Moreover we have the following.

\begin{proposition}\label{prop-inv-link}
Let $L$ and $L'$ be links. 
The following {\rm (1)} and {\rm (2)} hold: 
\begin{enumerate}
\item 
Let $n$ be a positive integer. 
If $L$ and $L'$ are $(2n+{\rm lh})$-equivalent,
then $\Delta_{L}^{(n)}(I)=\Delta_{L'}^{(n)}(I)$ and 
$\omu_{L}^{(n)}(I)=\omu_{L'}^{(n)}(I)$ for any non-repeated sequence $I$. 

\item
Let $p$ be a prime number.  
If $L$ and $L'$ are $2p$-move equivalent, then $\Delta_{L}^{(p)}(I)=\Delta_{L'}^{(p)}(I)$ and 
$\omu_{L}^{(p)}(I)=\omu_{L'}^{(p)}(I)$ for any sequence $I$ of length~$\leq p$. 
\end{enumerate}
\end{proposition}

\begin{proof}
Let $\sigma$ be a string link whose closure is $L$. 
By Lemma~\ref{lem-2nlh}, there exists a string link $\sigma'$ whose closure is $L'$ such that $\sigma'$ is $(2n+{\rm lh})$-equivalent to $\sigma$.  
By Theorem~\ref{th-invariance},  for any non-repeated sequence $I$, $\mu_{\sigma}(I)\equiv\mu_{\sigma'}(I)\pmod{n}$.   
Therefore, 
\[
\Delta_{L}^{(n)}(I)=\gcd{\{\Delta_{\sigma}(I),n\}}=\gcd{\{\Delta_{\sigma'}(I),n\}}=\Delta_{L'}^{(n)}(I).
\] 
Since $\Delta_{L}^{(n)}(I)$ divides $n$, it follows that 
\[
\mu_{\sigma}(I)\equiv\mu_{\sigma'}(I)\pmod{\Delta_{L}^{(n)}(I)}. 
\] 
This completes the proof of Proposition~\ref{prop-inv-link}~(1). 

Using Proposition~\ref{prop-inv-prime} instead of Theorem~\ref{th-invariance}, 
Proposition~\ref{prop-inv-link}~(2) is similarly shown. 
\end{proof}

Proposition~\ref{prop-inv-link}~(1) together with Theorem~\ref{th-sl} implies the following.

\begin{theorem}\label{th-link}
Let $n$ be a positive integer, and  
let $L$ and $L'$ be $m$-component links. 
Assume that $\Delta_{L}^{(n)}(I)=\Delta_{L'}^{(n)}(I)=n$ for any non-repeated sequence $I$ of length~$m$.
Then, $L$ and $L'$ are $(2n+{\rm lh})$-equivalent 
if and only if $\omu_{L}^{(n)}(I)=\omu_{L'}^{(n)}(I)$ for any non-repeated sequence~$I$ of length~$m$. 
\end{theorem}

\begin{proof}
Since the ``only if'' part directly follows from Proposition~\ref{prop-inv-link}~(1),  
it is enough to show the ``if'' part. 
Let $\sigma$ and $\sigma'$ be string links whose closures are $L$ and $L'$, respectively. 
Since $\Delta_{L}^{(n)}(I)=\Delta_{L'}^{(n)}(I)=n$ for any non-repeated sequence $I$ of length~$m$, 
it follows that 
\[
\mu_{\sigma}(J)\equiv\mu_{\sigma'}(J)\equiv0\pmod{n}
\] 
for any non-repeated sequence $J$ of length~$< m$. 
Furthermore, since $\omu_{L}^{(n)}(I)=\omu_{L'}^{(n)}(I)$ for any non-repeated sequence $I$ of length~$m$, we have
\[
\mu_{\sigma}(I)\equiv\mu_{\sigma'}(I)\pmod{n}.
\]
Therefore, $\sigma$ and $\sigma'$ are $(2n+{\rm lh})$-equivalent by Theorem~\ref{th-sl}. 
This completes the proof. 
\end{proof}

As a consequence of Theorem~\ref{th-link}, we have the following.

\begin{corollary}\label{cor-trivial} 
Let $n$ be a positive integer.
An $m$-component link $L$ is $(2n+{\rm lh})$-equivalent to the trivial link if and only if $\Delta_{L}^{(n)}(I)=n$ and 
$\omu_{L}^{(n)}(I)=0$ for any non-repeated sequence~$I$ of length~$m$. 
\end{corollary}

\begin{proof} 
This follows from Proposition~\ref{prop-inv-link}~(1) 
and Theorem~\ref{th-link}. 
\end{proof}


\end{document}